\documentclass[11pt]{amsart}

\usepackage{amsfonts, amsmath, amscd}
\usepackage[psamsfonts]{amssymb}

\usepackage{amssymb}

\usepackage[usenames]{color}

\newtheorem{theorem}{Theorem}[section]
\newtheorem{lemma}[theorem]{Lemma}
\newtheorem{corollary}[theorem]{Corollary}
\newtheorem{question}[theorem]{Question}
\newtheorem{remark}[theorem]{Remark}
\newtheorem{proposition}[theorem]{Proposition}
\newtheorem{definition}[theorem]{Definition}
\newtheorem{example}[theorem]{Example}
\newtheorem{fact}[theorem]{Fact}

\numberwithin{equation}{section}

\newcommand{\CC}{C_k}
\newcommand{\NN}{\mathbb{N}}

\newcommand{\UU}{\mathcal{U}}

\newcommand{\SP}{\mathrm{sp}}

\newcommand{\VV}{\mathbb{V}}

\newcommand{\Nn}{\mathcal{N}}

\newcommand{\IR}{\mathbb{R}}
\newcommand{\II}{\mathbb{I}}

\renewcommand{\phi}{\varphi}

\newcommand{\supp}{\mathrm{supp}}
\newcommand{\conv}{\mathrm{conv}}

\input xy
\xyoption{all}

\title[Free topological vector spaces]{Free topological vector spaces}

\author[S. S.~Gabriyelyan]{Saak S. Gabriyelyan}
\address{Department of Mathematics, Ben-Gurion University of the Negev, Beer-Sheva, P.O. 653, Israel}
\email{saak@math.bgu.ac.il}

\author[S. A. Morris]{Sidney A. Morris}
\address{Faculty of Science and Technology, Federation University Australia, PO Box 663, Ballarat, Victoria, 3353,  Australia \& Department of Mathematics and Statistics,  La~Trobe University, Melbourne, Victoria, 3086, Australia}
\email{morris.sidney@gmail.com}



\begin{document}

\begin{abstract}
We define and study the  free topological vector space $\VV(X)$ over a Tychonoff space $X$. We prove that $\VV(X)$ is a $k_\omega$-space if and only if    $X$ is a $k_\omega$-space. If $X$ is infinite, then $\VV(X)$ contains a closed  vector subspace which is topologically isomorphic to $\VV(\NN)$. It is proved that if $X$ is a $k$-space, then $\VV(X)$ is locally convex if and only if $X$ is discrete and countable. If $X$ is a metrizable space it is shown that: (1) $\VV(X)$ has countable tightness if and only if $X$ is separable, and (2) $\VV(X)$ is a $k$-space if and only if $X$ is locally compact and separable.  It is proved that $\VV(X)$ is a barrelled topological vector space if and only if $X$ is discrete.  This result is applied to free locally convex spaces $L(X)$ over a Tychonoff space $X$ by showing that: (1)  $L(X)$ is quasibarrelled if and only if $L(X)$ is barrelled if and only if $X$ is discrete, and (2) $L(X)$ is a Baire space if and only if  $X$ is finite.
\end{abstract}


\maketitle


\section{Introduction. }


Until recently almost all papers in topological vector spaces restricted themselves to locally convex spaces. However in recent years a number of questions about non-locally convex vector spaces have arisen.

All topological spaces are assumed here to be Tychonoff and all vector spaces are over the field of real numbers $\IR$.
The free  topological group $F(X)$, the free abelian topological group  $A(X)$ and the free locally convex space  $L(X)$ over a Tychonoff space $X$ were introduced by Markov \cite{Mar} and intensively studied over the last half-century, see for example \cite{ArT,Gra,Kakutani,MMO,Rai,Sipacheva}.
It has been known for half a century that the (Freyd) Adjoint Functor Theorem (\cite{MacLane} or Theorem A3.60 of \cite{HM}) implies the existence and uniqueness of  $F(X)$,  $A(X)$  and  $L(X)$.
This paper focuses on free topological vector spaces. One surprising fact is that  free topological vector spaces in some respect behave better than  free locally convex spaces.




\section{Basic properties of free topological vector spaces}


\begin{definition} \label{Def:FreeVS} {\em
The {\em free topological vector space} $\VV(X)$ over a Tychonoff space $X$ is a pair consisting of a topological vector space $\VV(X)$ and a continuous mapping $i=i_X: X\to \VV(X)$ such that every continuous mapping $f$ from $X$ to a topological vector space (tvs) $E$ gives rise to a unique continuous linear operator ${\bar f}: \VV(X) \to E$ with $f={\bar f} \circ i$.}
\end{definition}

In analogy with the Graev free abelian topological group  over a Tychonoff space $X$ with a distinguished point $p$, we can define the Graev free topological vector space $\VV_G(X,p)$ over $(X,p)$.
\begin{definition} \label{Def:FreeVS-Graev} {\em
The {\em Graev free topological vector space} $\VV_G(X,e)$ over a Tychonoff space $X$ with a distinguished point $e$ is a pair consisting of a topological vector space $\VV_G(X,e)$ and a continuous mapping $i=i_X: X\to \VV_G(X,e)$ such that $i(e)=0$ and every continuous mapping $f$ from $X$ to a topological vector space $E$ with $f(e)=0$ gives rise to a unique continuous linear operator ${\bar f}: \VV_G(X,e) \to E$ with $f={\bar f} \circ i$.}
\end{definition}

We shall use the notation: for a subset $A$ of a vector space $E$ and a natural number $n\in\NN$ we denote by $\SP_n(A)$ the following subset of $E$
\[
\SP_n(A):= \{ \lambda_1 x_1 +\cdots  + \lambda_n x_n : \; \lambda_i \in [-n,n], \;  x_i\in A, \; \forall i=1,\dots,n\},
\]
and set $\SP(A):=\bigcup_{n\in\NN} \SP_n(A)$, the span of $A$ in $E$.

As $X$ is a Tychonoff space, the mapping $i_X$ is an embedding. So we identify the space $X$ with $i(X)$ and regard $X$ as a subspace of $\VV(X)$.
\begin{theorem} \label{t:Free-exists}
Let $X$ be a  Tychonoff space and $e\in X$ a distinguished point. Then
\begin{enumerate}
\item[{\rm (i)}] $\VV(X)$ and  $\VV_G(X,e)$ exist (and are Hausdorff);
\item[{\rm (ii)}] $\SP(X)=\VV(X)$ and  $X$ is a vector space basis for $\VV(X)$;
\item[{\rm (iii)}] $\SP(X)=\VV_G(X,e)$ and  $X\setminus\{ e\}$ is a vector space basis for $\VV_G(X,e)$;
\item[{\rm (iv)}] $\VV(X)$  and  $\VV_G(X,e)$ are unique up to isomorphism of topological vector spaces;
\item[{\rm (v)}] $X$ is a closed subspace of $\VV(X)$  and  $\VV_G(X,e)$;
\item[{\rm (vi)}]  $\SP_n(X)$ is closed in $\VV(X)$  and  $\VV_G(X,e)$, for every $n\in\NN$;
\item[{\rm (vii)}] if $q:X\to Y$ is a quotient map of Tychonoff spaces $X$ and $Y$, then  $\VV(Y)$ is a quotient topological vector space of $\VV(X)$;
\item[{\rm (viii)}] if $Y$ is a Tychonoff space with a distinguished point $p$ and $X\wedge Y$ is the wedge sum of $(X,e)$ and $(Y,p)$, then  $\VV_G(X,e)\times \VV_G(Y,p) =\VV_G\big(X\wedge Y,(e,p)\big)$.
\end{enumerate}
\end{theorem}

\begin{proof}
(i)-(iv) follow from the Adjoint Functor Theorem.



(v)-(vi) We consider only the case $\VV(X)$. Let $\beta X$ be the Stone-\v{C}ech compactification of $X$. Then, by the definition of free topological vector space, the natural map $\beta:X\to \beta X \subseteq \VV(\beta X)$ can be extended to a continuous injective linear operator ${\bar \beta}:\VV(X) \to \VV(\beta X)$. Since $\beta X$ and $\SP_n(\beta X)$ are compact subsets of $\VV(\beta X)$, $X={\bar \beta}^{-1}(\beta  X)$  and $\SP_n(X) ={\bar \beta}^{-1} \big( \SP_n(\beta X) \big)$ by the injectivity of ${\bar \beta}$, we obtain that $X$ and $\SP_n(X)$ are closed subsets of $\VV(X)$, for every $n\in\NN$.

(vii) Let ${\bar q}: \VV(X) \to \VV(Y)$ be a continuous linear operator extending $q$.  Set $H=\ker({\bar q})$ and let $j:\VV(X)\to \VV(X)/H$ be the quotient map. Denote by $f:\VV(X)/H \to \VV(Y)$ the induced continuous linear map. We have to show that the topology of $\VV(X)/H$ coincides with the topology of $\VV(Y)$.

Let $E$ be an arbitrary tvs and $t: Y\to E$ a continuous map. Then $q\circ t:X\to E$ is continuous. So there is a unique continuous extension $\overline{q\circ t}: \VV(X) \to E$. As $\VV(X)/H$ is algebraically $\VV(Y)$, we obtain the induced linear map $T: \VV(X)/H \to E$. Now if $U$ is open in $E$, then $V:=\overline{q\circ t}^{-1} (U)$ is open in $\VV(X)$ and hence $j(V)$ is open in $\VV(X)/H$. Since $T\big( j(V)\big)=\overline{q\circ t}(V)=U$, we obtain that $T$ is continuous. Finally the definition of free tvs implies that $\VV(X)/H$ is $\VV(Y)$.

(viii)  follows from  the Adjoint Functor Theorem since the left adjoint functor preserves (finite) coproducts.
\end{proof}
We shall  denote the topology of $\VV(X)$  by $\pmb{\mu}_X$. So $\VV(X) =(\VV_X, \pmb{\mu}_X)$, where $\VV_X$ is a vector space with a basis $X$:
\[
\VV_X :=\{ \lambda_1 x_1 +\cdots + \lambda_n x_n :\; n\in\NN, \lambda_i \in \IR, x_i\in X\}.
\]

If $v\in \VV(X)$ (or $v\in L(X)$) has a representation
\[
v=\lambda_1 x_1 +\cdots + \lambda_n x_n, \mbox{ where } \lambda_i \in \IR\setminus \{ 0\} \mbox{ and } x_i\in X \mbox{ are distinct},
\]
the set $\supp(v):=\{ x_1,\dots,x_n\}$ is called the {\em support} of the element $v$.





For a subspace $Z$ of a space $X$, let $\VV(Z,X)$ be the vector subspace of $\VV(X)$ generated algebraically by $Z$.
\begin{lemma} \label{l:Free-subspace}
Let $Z$ be a closed subspace of a space $X$. Then $\VV(Z,X)$ is a closed subspace of $\VV(X)$.
\end{lemma}
\begin{proof}
Assume that $h=\sum_{i=1}^n a_i x_i$ does not belong to $\VV(Z,X)$, where $a_i \not= 0$ and $x_i$ are distinct for all $i$. Then there is an index $i$, say $i=1$, such that $x_1 \not\in Z$. Since $X$ is Tychonoff, there is a function $f:X\to \IR$ such that $f(x_1)=1$ and $f\big( Z\cup \{ x_2,\dots,x_n\} \big) =0$. Lift this function to a linear mapping ${\bar f}: \VV(X) \to \IR$. Now ${\bar f}\big(\VV(Z,X)\big)=0$ and ${\bar f}(h)=a_1 f(x_1)=a_1 \not=0$. If $U$ is an open neighborhood of $a_1$ not containing zero, then ${\bar f}^{-1}(U)$ is an open neighborhood of $h$ which does not intersect $\VV(Z,X)$.
\end{proof}

\begin{proposition} \label{p:Free-retract}
If $X$ is a Tychonoff space and $Z$ is a retract (in particular, a clopen subset) of $X$, then $\VV(Z)$ embeds onto a closed vector subspace of $\VV(X)$.
\end{proposition}
\begin{proof}
Let $p:X\to Z$ be a retraction. Then $\VV(Z)$ is a quotient topological vector space of $\VV(X)$ under an extension ${\bar p}$ of $p$ by Theorem \ref{t:Free-exists}(vii). As $p(z)=z$ we obtain that ${\bar p}$ is injective on $\VV(Z,X)$ and ${\bar p}(\VV(Z,X))=\VV(Z)$. So the topology $\tau$ of $\VV(Z,X)$ is finer than the topology $\pmb{\mu}_Z$ of $\VV(Z)$. Now the definition of $\pmb{\mu}_Z$ implies that $\tau = \pmb{\mu}_Z$. Thus $\VV(Z)$ embeds onto $\VV(Z,X)$. Finally $\VV(Z,X)$ is a closed vector subspace of $\VV(X)$ by Lemma \ref{l:Free-subspace}.
\end{proof}

\begin{corollary} \label{c:Free-retract}
If $X= Y \cup Z$ is a disjoint union of Tychonoff  spaces $Y$ and $Z$, then $\VV(X) = \VV(Y)\oplus \VV(Z)$.
\end{corollary}
Below we generalize Corollary \ref{c:Free-retract}. First we recall some definitions.

For a non-empty family $\{ E_i \}_{i\in I}$  of vector spaces, the \emph{direct sum} of $E_i$ is denoted by
\[
\bigoplus_{i\in I} E_i :=\left\{ (x_i)_{i\in I} \in \prod_{i\in I} E_i : \; x_i = 0_i \mbox{ for all but a finite number of } i \right\},
\]
and we  denote by $j_k $ the natural embedding of $E_k$ into $\bigoplus_{i\in I} E_i$; that is,
\[
j_k (x)=(x_i)\in \bigoplus_{i\in I} E_i, \mbox{ where } x_i =x  \mbox{ if } i=k \mbox{ and } x_i =0_i  \mbox{ if } i\not= k.
\]
If $\{ E_i \}_{i\in I}$ is  a non-empty family  of topological vector spaces {\it the final vector space topology} $\mathcal{T}_f$ on $\bigoplus_{i\in I} E_i$  with respect to the family of canonical homomorphisms $j_k : E_k \to \bigoplus_{i\in I} E_i$  is the finest vector space topology on $\bigoplus_{i\in I} E_i$ such that all $j_k$ are continuous.

\begin{definition} \label{defCoproduct}{\em
Let $\mathcal{E}=\{ (E_i, \mathcal{T}_i) \}_{i\in I}$ be a non-empty  family of topological vector spaces. The topological vector space $(E,\mathcal{T})$ is the  \emph{coproduct} of the family $\mathcal{E}$ in the category  $\mathbf{TVS}$ of topological vector spaces and continuous linear operators if
\begin{enumerate}
\item[{\rm (i)}] for each $i\in I$ there is an embedding $j_i: E_i \to E$;
\item[{\rm (ii)}] for any tvs $V$ and each family $\{ p_i \}_{i\in I}$ of continuous linear mappings $p_i: E_i \to V$, there exists a unique continuous linear mapping $p:E\to V$ such that $p_i =p\circ j_i$ for every $i\in I$.
\end{enumerate}}
\end{definition}
The underlying vector space structure of the coproduct $(E,\mathcal{T})$  is the direct sum  $\bigoplus_{i\in I} E_i$. The  \emph{coproduct topology}  $\mathcal{T}$ on $E$ coincides with  the final  vector space  topology $\mathcal{T}_f$  with respect to the family of canonical homomorphisms $j_i : E_i\to E$. Note that a coproduct of a family of tvs is  unique up to topological linear isomorphism. If $I$ is countable, then the coproduct topology $\mathcal{T}$ on $E$ is the subspace topology on $E$ induced by the box topology on the product  $\prod_{i\in I} E_i$.

\begin{proposition} \label{p:Free-coproduct}
Let $X=\bigcup_{i\in I} X_i$ be a disjoint sum of a nonempty family $\{ X_i: i\in I\}$ of  Tychonoff  spaces. Then $\VV(X)$ is topologically isomorphic with the coproduct $(E,\mathcal{T})$ of the family $\{ \VV(X_i): i\in I\}$. If the set $I$ is countable, then the topology $\mathcal{T}$ is the subspace topology on the direct sum induced by the box topology on the product $\prod_{i\in I} \VV(X_i)$.
\end{proposition}
\begin{proof}
It is clear that the underlying vector spaces of $\VV(X)$ and $(E,\mathcal{T})$ is the direct sum  $\VV_X =\bigoplus_{i\in I} \VV_{X_i}$. Let $id_{\VV_X}: (E,\mathcal{T}) \to \VV(X)$ be the identity map and note that the inclusion $p_i: \VV(X_i) \to \VV(X)$ is an embedding by Proposition \ref{p:Free-retract}, for every $i\in I$. So, by the definition of coproduct topology, the map $id_{\VV_X}$ is continuous. Now the definition of the free  vector space  topology shows that $\pmb{\mu}_X =\mathcal{T}$.
\end{proof}

The next proposition shows that the Graev free topological vector space does not depend on the distinguished point. This we prove analogously to Theorem 2 of \cite{Gra}.
\begin{proposition} \label{p:Free-Graev-1}
Let $X$ be a Tychonoff space and $e,p \in X$ be distinct points. Then $\VV_G(X,e)$ and $\VV_G(X,p)$ are topologically isomorphic.
\end{proposition}

\begin{proof}
Define $\phi:(X,e)\to \VV_G(X,p)$ setting $\phi(x):=x-e$, where ``$-$'' denotes difference in $\VV_G(X,p)$. Then $\phi$ is continuous and $\phi(e)=0$. So there is a continuous linear map ${\bar \phi}: \VV_G(X,e) \to \VV_G(X,p)$ which extends $\phi$. Analogously, we define  $\psi:(X,p)\to \VV_G(X,e)$ setting $\psi(x):=x-p$. Then $\psi$ is continuous and $\psi(p)=0$. So there is a continuous linear map ${\bar \psi}: \VV_G(X,p) \to \VV_G(X,e)$ which extends $\psi$. Now, for every $x\in (X,e)$, we have
\[
{\bar \psi} {\bar \phi} (x)={\bar \psi}(x-e)={\bar \psi}(x) -{\bar \psi}(e)= \psi(x)-\psi(e)=(x-p)-(e-p)=x-e=x,
\]
since $e$ is the identity in the space $\VV_G(X,e)$. Since $X$ generates $\VV_G(X,e)$, we obtain that ${\bar \psi} {\bar \phi}$ is the identity map of $\VV_G(X,e)$. Analogously, ${\bar \phi} {\bar \psi}$ is the identity map of $\VV_G(X,p)$. Thus $\VV_G(X,e)$ and $\VV_G(X,p)$ are topologically isomorphic.
\end{proof}
So we can write $\VV_G(X,e)=\VV_G(X)$.

Our next result is analagous to Theorem 3 of \cite{Morris}.

\begin{proposition} \label{p:Free-Graev-2}
Let $X$ be a Tychonoff space and $\{ e\}$ be a singleton. Then
\[
\VV(X)=\VV_G\big(X\sqcup \{ e\}\big)=\VV_G(X)\times \IR.
\]
\end{proposition}

\begin{proof}
Let $\phi: X\to E$ be a continuous map into a  topological vector space $E$. Define $i:X\to X\sqcup \{ e\}$ by $i(x):=x$ and $\psi: X\sqcup \{ e\} \to E$ by $\psi(x):=x$ and $\psi(e):=0$. Then we obtain the following commutative diagram
\[
\xymatrix{
X  \ar@{->}[d]_i \ar@{->}[r]^\phi & E \\
X \sqcup \{ e\} \ar@{->}[ru]^\psi \ar@{->}[r] & \VV_G\big( X\sqcup \{ e\}, e\big) \ar@{->}[u]^{{\bar \psi}} }
\]
where $\bar \psi$ is a linear continuous extension of $\psi$. Therefore the space $H:=\VV_G\big( X\sqcup \{ e\}, e\big)$ has the universal property, and hence $\VV(X)=H$ by the uniqueness of the free tvs over $X$. This proves the first equality.

Fix arbitrarily a point $p\in X$.
The second equality follows from the first equality, Theorem  \ref{t:Free-exists}, Proposition \ref{p:Free-Graev-1} and the following chain of equalities
\[
\begin{split}
\VV_G(X,p)\times\IR & = \VV_G(X,p)\times \VV_G(\{ e,p\}, p) =\VV_G \big( (X,p)\wedge (\{ e,p\}, p) \big) \\
& =\VV_G\big(X\sqcup \{ e\}, p\big)=\VV_G\big(X\sqcup \{ e\}, e\big)=\VV(X).
\end{split}
\]
\end{proof}

Graev \cite{Gra} used this to show that non-homeomorphic Tychonoff spaces $X$ and $Y$ may have isomorphic the free topological groups $A(X)$ and $A(Y)$. The same holds for free topological vector spaces.
\begin{example} {\em
Let $X=[0,1]$, $Y=[1,2]$, $e=1/2\in X$ and $p=1\in X\cap Y$. Then $(X,p)\wedge (Y,p)=([0,2],p)$ is an interval which is not homeomorphic to the space $Z:=(X,e)\wedge (Y,p)$. However, the Graev free topological vector spaces $\VV_G([0,2])$ and $\VV_G\big( Z\big)$ are topologically isomorphic by Theorem  \ref{t:Free-exists} and Proposition \ref{p:Free-Graev-1} since
\[
\VV_G([0,2])=\VV_G(X,p)\times \VV_G(Y,p)= \VV_G(X,e) \times \VV_G(Y,p)=\VV_G(Z).
\]   }
\end{example}


\section{Free topological vector spaces over $k_\omega$-spaces}


Let $\{ (X_n,\tau_n )\}_{n\in\NN}$ be a sequence of topological spaces such that $X_n \subseteq X_{n+1}$ and $\tau_{n+1} |_{X_n} = \tau_n$ for all $n\in\NN$.  The union $X=\cup_{n\in\NN} X_n$ with the weak topology $\tau$ (i.e., $U\in\tau$ if and only if $U\cap X_n \in\tau_n$ for every $n\in\omega$) is called the {\it inductive limit} of the sequence $\{ (X_n,\tau_n )\}_{n\in\NN}$ and it is denoted by $(X,\tau)= \underrightarrow{\lim} (X_n,\tau_n )$. Recall that a topological space is called a $k_\omega$-{\it space} if it is the inductive limit of an increasing sequence of its compact subsets. A topological group $(G,\tau)$ is called a $k_\omega$-{\it group} if its underlying topological space is a $k_\omega$-space.

\begin{theorem} \label{t:Free-k-space}
If $X$ is a $k_\omega$-space and $X=\cup_{n\in\NN} C_n$ is a $k_\omega$-decomposition of $X$, then $\VV(X)$ is a $k_\omega$-space and $\VV(X)= \cup_{n\in\NN} \SP_n(C_n)$ is a $k_\omega$-decomposition of $\VV(X)$.
\end{theorem}

\begin{proof}
Let $(X,\tau)= \underrightarrow{\lim} (C_n,\tau_n )$, where $C_1\subseteq C_2\subseteq\dots$ are compact. For every $n\in\NN$ denote by $X_n$ the image of the mapping $T_n :[-n,n]^n\times C_n^n \to \VV(X)$ defined by
\[
T_n \big( (a_1,\dots,a_n)\times(x_1,\dots,x_n)\big) := a_1 x_1 +\cdots +a_n x_n.
\]
Since $\VV(X)$ is a tvs, $T_n$ is continuous and hence $X_n = \SP_n(C_n)$ is a compact subspace of $\VV(X)$. Clearly, $\VV(X)=\cup_{n\in\NN} X_n$.

Let $\tau_n$ be the topology of the compact space $X_n$.  Clearly, $\tau_{n+1}|_{X_n} = \tau_n$. So we can define the inductive limit $(\VV_X,\tau) =\underrightarrow{\lim} (X_n,\tau_n)$. Then $(\VV_X,\tau)$ is a $k_\omega$-space and every  compact subset of $(\VV_X,\tau)$ is contained in some $X_n$ by \cite[Lemma 9.3]{Ste}. It is clear that $\pmb{\mu}_X \leq \tau$. So to prove the proposition it is enough to show that $\tau$ is a vector space topology on $\VV_X$.

To  this end we must prove that the map
\[
T: \IR \times (\VV_X,\tau)\times (\VV_X,\tau) \to (\VV_X,\tau), \quad T(\lambda,x,y):= \lambda x +y,
\]
is continuous. Since $(\VV_X,\tau)$ is a $k_\omega$-space, the space $Z:=\IR \times (\VV_X,\tau)\times (\VV_X,\tau)$ is also a $k_\omega$-space. Thus, to show that $T$ is continuous we only have to show that $T$ is continuous on all compact subsets of $Z$.

Let $K$ be a compact subset of $Z$. Then $K\subseteq [-n,n]\times X_n\times X_n$ for some $n\in\NN$. Thus
\[
T(K)\subseteq T\big( [-n,n]\times X_n\times X_n\big) \subseteq X_{n^2 +n}.
\]
Noting that $K$ is compact and $\pmb{\mu}_X \leq \tau$, we see that $K$ has the same induced topology as a subset of $Z$ as it has as a subset of $\IR \times \VV(X)\times \VV(X)$. Since $\tau|_{X_{n^2 +n}} =\pmb{\mu}_X|{X_{n^2 +n}}$ and $\pmb{\mu}_X$ is a vector space topology, $T:K\to X_{n^2 +n}$ is continuous. So $T$ is continuous and hence $(\VV_X,\tau)$ is a topological vector space. Thus $\tau=\pmb{\mu}_X$ by the definition of $\pmb{\mu}_X$.
\end{proof}
\begin{remark} {\em
Banakh proved Theorem \ref{t:Free-k-space} for the special case of $X$ being a submetrizable $k_\omega$-space. The research  in \cite{Banakh-Survey} and the research  in this paper  were done independently.}
\end{remark}

 The (Weil) completeness is one of the most important properties of topological groups. Recall that any $k_\omega$ topological group is complete by \cite[Theorem 2]{HuntMorris}.
\begin{corollary} \label{c:Free-complete}
If $X$ is a $k_\omega$-space, then $\VV(X)$ is  complete.
\end{corollary}

\begin{corollary} \label{c:Free-compact}
Let $X$ be a Tychonoff space. Then for every compact subset $K$ of $\VV(X)$ there is an $n\in\NN$ such that $K\subseteq \SP_n(X)$.
\end{corollary}
\begin{proof}
Denote by $\beta X$ the Stone--\v{C}ech compactification of $X$, and let ${\bar \beta}:\VV(X)\to \VV(\beta X)$ be an extension of $\beta$. Consider the following sequence
\[
\xymatrix{
K  \ar@{->}[r]^{id_K} & \VV(X) \ar@{->}[r]^{\bar \beta} & \VV(\beta X). }
\]
Then, by Theorem \ref{t:Free-k-space}, there is an $n\in\NN$ such that ${\bar \beta}(K) \subseteq \SP_n(\beta X)$. Since ${\bar \beta}$ is injective we obtain $K\subseteq {\bar \beta}^{-1} \big( \VV(X)\cap \SP_n(\beta X)\big) = \SP_n(X)$.
\end{proof}

Theorem \ref{t:Free-k-space} is surprising since it contrasts with what is known about free locally convex spaces, see Fact \ref{f:Free-L}.
\begin{definition} {\em
The {\em  free locally convex space} $L(X)$ over a Tychonoff space $X$ is a pair consisting of a locally convex space $L(X)$ and  a continuous mapping $i: X\to L(X)$ such that every  continuous mapping $f$ from $X$ to a locally convex space  (lcs) $E$ gives rise to a unique continuous linear operator ${\bar f}: L(X) \to E$  with $f={\bar f} \circ i$. }
\end{definition}
We shall denote the topology of $L(X)$ by $\pmb{\nu}_X$, so $L(X)=(\VV_X,\pmb{\nu}_X)$.

\begin{fact}[\cite{Gabr}] \label{f:Free-L}
For a Tychonoff space $X$, the space $L(X)$ is a $k$-space if and only if $X$ is a countable discrete space.
\end{fact}

Let us recall also the definition of free abelian topological groups.
\begin{definition} {\em
Let $X$ be a Tychonoff space. An abelian topological group $A(X)$ is called {\em  the free abelian topological  group} over  $X$ if $A(X)$ satisfies the following conditions:
\begin{enumerate}
\item[{\rm (i)}] there is a continuous mapping  $i: X\to A(X)$ such that $i(X)$ algebraically generates $A(X)$;
\item[{\rm (ii)}] if $f: X\to G$ is a continuous mapping to an abelian topological  group $G$, then there exists a continuous homomorphism ${\bar f}: A(X) \to G$  such that $f={\bar f} \circ i$.
\end{enumerate} }
\end{definition}
The fact that $\VV(X)$ is complete when $X$ is a $k_\omega$-space contradicts  the known result (Fact \ref{f:Free-Completness}) that the free locally convex space $L(X)$ is not complete if $X$ has any infinite compact subspace.
\begin{fact} \label{f:Free-Completness}
Let $X$ be a Tychonoff space. Then
\begin{enumerate}
\item[{\rm (i)}] {\em \cite{Tkac}} $A(X)$ is complete if and only if $X$ is Dieudonn\'{e} complete;
\item[{\rm (ii)}] {\em  \cite{Usp2}} $L(X)$ is complete if and only if $X$ is Dieudonn\'{e} complete and does not have infinite compact subsets.
\end{enumerate}
\end{fact}

\begin{proposition} \label{p:Free-MMO-1}
Let $X=\bigcup_{n\in\NN} C_n$ be a $k_\omega$-decomposition of a $k_\omega$-space $X$ into compact sets and let $E$ be a topological vector space generated algebraically by $X$. Further, let $E$ have the property that a subset $A$ of $E$ is closed in $E$ if and only if $A\cap \SP_n(C_n)$ is compact for all $n\in\NN$. Then the topology of $E$ is the finest vector topology which induces the given topology on $X$, and so $E$ is the free topological vector space over $X$.
\end{proposition}
\begin{proof}
Let $\tau$ be the given topology on $E$ and $\tau' \supseteq \tau$ the finest vector topology inducing the given topology on $X$. By the proof of Theorem \ref{t:Free-k-space}, $A\subseteq E$ is closed in $\tau'$ if and only if  each $A\cap \SP_n(C_n)$ is compact. But $\tau$ and $\tau'$ induce the same topology on $X$, hence on $C_n$ and hence also on $\SP_n(C_n)$. Thus $\tau' =\tau$ as desired.
\end{proof}

\begin{proposition} \label{p:Free-MMO-2}
Let $X=\cup_{n\in\NN} C_n$ be a $k_\omega$-space and let $Y$ be a subset of $\VV(X)$  such that $Y$ is a free vector space basis for the subspace, $\SP(Y)$, that it generates. Assume that $K_1,K_2,\dots$ is a sequence of compact subsets of $Y$ such that $Y=\cup_{n\in\NN} K_n$ is  a $k_\omega$-decomposition of $Y$ inducing the same topology on $Y$ that $Y$ inherits as a subset of $\VV(X)$. If for every $n\in\NN$ there is a natural number $m$ such that $\SP(Y)\cap \SP_n(C_n) \subseteq \SP_m(K_m)$, then $\SP(Y)$ is $\VV(Y)$, and both $\SP(Y)$ and $Y$ are closed subsets of $\VV(X)$.
\end{proposition}
\begin{proof}
It follows from the proof of Theorem \ref{t:Free-k-space} that, to prove $\SP(Y)$ is closed in $\VV(X)$, we only have to show that $\SP(Y)\cap \SP_n(C_n)$ is compact for each $n\in\NN$. Now
\[
\SP(Y)\cap \SP_n(C_n)=\SP(Y)\cap \SP_n(C_n)\cap \SP_m(K_m)=\SP_n(C_n)\cap \SP_m(K_m),
\]
and hence is compact. Thus $\SP(Y)$ is closed in $\VV(X)$. Analogously, $Y$ is closed in $\VV(X)$.

Using Proposition \ref{p:Free-MMO-1}, to prove that $\SP(Y)$ is the free topological vector space on $Y$, it suffices to show that a subset $A$ of $\SP(Y)$ is closed if $A\cap \SP_n(K_n)$ is compact for all $n\in\NN$. Consider $A\cap \SP_n(C_n)$, for any $n$. Then there is an $m\in\NN$ such that $\SP(Y)\cap \SP_n(C_n) \subseteq \SP_m(K_m)$ and hence
\[
A\cap \SP_n(C_n) =A\cap \SP_n(C_n) \cap \SP(Y)=A\cap \SP_n(C_n)\cap \SP_m(K_m)= \big(A\cap \SP_m(K_m)\big) \cap \SP_n(C_n).
\]
Since both $A\cap \SP_m(K_m)$ and $\SP_n(C_n)$ are compact,  $A\cap \SP_n(C_n)$ is compact in $\VV(X)$. Thus $A$ is a closed subset of $\VV(X)$ and the proof is complete.
\end{proof}

Since a closed subspace of a $k_\omega$-space is also a $k_\omega$-space, Lemma \ref{l:Free-subspace} and Proposition \ref{p:Free-MMO-2} implies
\begin{corollary} \label{c:Free-1}
If $Y$ is a closed subspace of a $k_\omega$-space $X$, then the closed subspace $\VV(Y,X)$ of $\VV(X)$ is $\VV(Y)$.
\end{corollary}

\begin{proposition} \label{p:Free-Subspace}
If $K$ is a compact subspace of a Tychonoff space $X$, then $\VV(K,X)$ is $\VV(K)$.
\end{proposition}
\begin{proof}
Denote by $\beta X$ the Stone--\v{C}ech compactification of $X$. So we obtain the following commutative diagram
\[
\xymatrix{
K  \ar@{->}[d] \ar@{->}[r]^i & X \ar@{->}[d]\ar@{->}[r]^\beta & \beta X \ar@{->}[d]\\
\VV(K) \ar@{->}[r]^{{\bar i}} & \VV(X) \ar@{->}[r]^{{\bar \beta}} & \VV(\beta X) } .
\]
Then $\VV(K,\beta X)=\VV(K)$ by Corollary \ref{c:Free-1}. So we obtain
\[
\VV(K) \stackrel{\bar i}{\longrightarrow}  {\bar i}\big( \VV(K)\big)=\VV(K,X) \stackrel{\bar \beta}{\longrightarrow}  \VV(K) ,
\]
and so $\VV(K)=\VV(K,X)$.
\end{proof}

\begin{proposition} \label{p:Free-Metrizable}
If $X$ is a $k_\omega$-space, then every metrizable (in particular, Banach) vector subspace $E$ of $\VV(X)$ is finite-dimensional.
\end{proposition}

\begin{proof}
Let $\bar{E}$ be the closure of $E$ in $\VV(X)$. Then $\bar{E}$ is a metrizable closed subspace of $\VV(X)$. As $\VV(X)$ is a $k_\omega$-space by Theorem \ref{t:Free-k-space}, $\bar{E}$ is also a $k_\omega$-space. Therefore $\bar{E}$ is a locally compact by \cite[3.4.E]{Eng}, and hence it is finite-dimensional, see \cite[\S 15.7]{Kothe}.  Thus $E$ is finite-dimensional as well.
\end{proof}

On the other hand, we now see that every infinite-dimensional space $\VV(X)$ contains the space $\phi =\VV(\NN)$. But $\phi$ is the inductive limit of $\IR^n$, and so $\phi$ is a locally convex space and therefore $\phi=L(\NN)$.

For every $n\in \NN$, set
\begin{equation} \label{equation-1}
T_n :=1+\cdots +n \quad \mbox{ and } \quad S_n :=T_1+\cdots +T_n.
\end{equation}
\begin{theorem} \label{t:Free-V(N)}
If $X$ is an infinite Tychonoff space, then $\VV(X)$ contains a closed vector subspace which is topologically isomorphic to $\phi= \VV(\NN)=L(\NN)$.
\end{theorem}

\begin{proof}
First we assume that $X$ is an infinite compact space. Take arbitrarily a sequence $\{ z_n\}_{n\in\NN}$ of distinct elements of $X$. For every $n\in \NN$ and $S_n$ defined in (\ref{equation-1}), set
\[
\begin{split}
y_1 & := z_1, \quad y_2 := 2z_1 + z_2 + z_3, \\
y_n & := nz_1 + z_{S_{n-1}+1} + z_{S_{n-1}+2} +\cdots + z_{S_{n}}, \quad n>2,
\end{split}
\]
where the $S_n$ are as in (\ref{equation-1}) and ``$+$'' denotes the vector space addition in $\VV(X)$. Since the sequence $\{ nz_1\}_{n\in\NN}$ is discrete and closed in $\IR z_1 \subset\VV(X)$, the sequence $Y:= \{ y_n\}_{n\in\NN}$ is discrete and closed in $\VV(X)$. So $\SP(Y)=\VV(Y,X)$ is a closed vector subspace in $\VV(X)$ by Lemma \ref{l:Free-subspace}. Let us show that $\SP(Y)$ is topologically isomorphic to $\phi$. For every $n\in\NN$, set $K_n :=\{ y_1,\dots,y_n\}$.

 We claim that $\SP(Y) \cap \SP_n(X) \subseteq \SP_n(K_n)$. Indeed, fix $t\in \SP(Y) \cap \SP_n(X)$.  So there are distinct $x_1,\dots,x_n\in X$, $i_1<\dots<i_m$, nonzero real numbers $a_1,\dots,a_m$ and nonzero numbers $\lambda_1,\dots, \lambda_n \in [-n,n]$ such that
\[
\begin{split}
t & =a_1 y_{i_1} +\cdots + a_m y_{i_m} \\
& =\big(a_1  n_{i_1} +\cdots + a_m n_{i_m}\big)z_1 + \sum_{l=1}^m \sum_{j=1}^{i_l} a_l z_{S_{i_l -1}+j} \\
& = \lambda_1 x_1 +\cdots +\lambda_n x_n.
\end{split}
\]
But since all the elements $z_{S_{i_l -1}+j}$ are distinct elements of the canonical basis of $\VV(X)$, this equality implies that $i_m\leq n$ and $\{ a_1,\dots,a_m\} \subseteq \{ \lambda_1,\dots, \lambda_n\}$. Thus $t\in \SP_n(K_n)$.

The topology $\tau$ of $\SP(Y)$ induced from $\VV(X)$ is defined by the sequence $\{ \SP(Y) \cap \SP_n(X)\}_{n\in\NN}$ of compact sets by Theorem \ref{t:Free-k-space} and closedness of $\SP(Y)$. Note that for every $n\in \NN$ there is an  $m\in\NN$ such that $\SP_n(K_n) \subseteq \SP(Y) \cap \SP_m(X)$. This inclusion and the claim imply that $\tau$ coincides with  the free topology $\pmb{\mu}_Y$ on $\SP(Y)$ defined by the sequence $\{ \SP_n(K_n)\}_{n\in\NN}$.  Thus  $\SP(Y)$  is topologically isomorphic to $\phi$.

Now let the space $X$ be arbitrary and let $\beta X$ be the Stone-\v{C}ech compactification of $X$. So there is a continuous linear monomorphism ${\bar i}: \VV(X) \to \VV(\beta X)$. As $X$ is infinite, $\beta X$ contains a sequence $\{ z_n\}_{n\in\NN}$ of distinct elements of $X$. So $\VV(\beta X)$ contains a closed and discrete subset $Y$ such that $\SP(Y)$ is a closed vector subspace of $\VV(\beta X)$ which is topologically isomorphic to $\phi$ by the first step. Clearly, ${\bar i}^{-1}(Y)$ is a closed and discrete subset of $\VV(X)$ and $E:= {\bar i}^{-1}\big(\SP(Y)\big)$ is a closed vector subspace of $\VV(X)$. So the topology $\tau'$ on $E$ induced from $\VV(X)$ is finer than the free topology $\pmb{\mu}_Y$ on $E$. Therefore $\tau' =\pmb{\mu}_Y$. Thus $\phi$ is  topologically isomorphic to the closed vector subspace $E$ of $\VV(X)$.
\end{proof}
Theorem \ref{t:Free-V(N)} is of interest also because such a result does not hold for free locally convex spaces. Indeed,  Theorem 4.3 of \cite{LMP} states that if the free locally convex space $L(X)$ over a Tychonoff space $X$ embeds into $L(\II)$, where $\II=[0,1]$, then $X$ is a metrizable countable-dimensional compactum. In particular, the space $\phi=L(\NN)$ does not embed into $L(\II)$.

\begin{theorem} \label{t:Free-V[0,1]}
Let $\{ K_n\}_{n\in\NN}$ be a sequence of disjoint compact subsets of $\IR$. Then $\VV\big(\bigsqcup_{n\in\NN} K_n\big)$ embeds onto a closed vector subspace of $\VV(\II)$.
\end{theorem}

\begin{proof}
Take two sequences $\{ a_n\}_{n\in\NN},\{b_n\}_{n\in\NN} \subset \II$ such that $a_1=0$ and $a_n<b_n<a_{n+1} <1$ for every $n\in\NN$. For every $n\in\NN$, set $I_n :=[a_n,b_n]$ and $X_n :=\bigcup_{i\leq n} I_i$. Since each $K_n$ is homeomorphic to a closed subset of $I_n$, by Corollary \ref{c:Free-1}, we can assume that $K_n =I_n$ for every $n\in\NN$. Set $X:= \bigsqcup_{n\in\NN} I_n$.

For every $n\in\NN$ and $1\leq i\leq A_n :=[S_n/2]$, the integer part of $S_n/2$ as defined in (\ref{equation-1}), we define the closed interval
\[
I_{i,n} := \left[ a_n + \frac{b_n -a_n}{S_n} (2i-1), a_n + \frac{b_n -a_n}{S_n} 2i\right],
\]
and define the homeomorphism $g_{i,n}: I_n \to I_{i,n}$ by
\[
g_{i,n}(x):=  \frac{1}{S_n} x + a_n \left( 1- \frac{1}{S_n}\right) +  \frac{b_n -a_n}{S_n}(2i-1).
\]
For every $n\in\NN$, define the map $h_n: I_n \to \VV(\II)$ by
\[
h_n(x) := g_{1,n}(x) + g_{2,n}(x)+\cdots + g_{A_n,n}(x),
\]
where ``$+$'' denotes the vector space addition in $\VV(\II)$. Now we define the map $\chi: X\to \VV(\II)$ setting
\[
\chi(x):= h_n(x), \quad \mbox{ if } x\in I_n.
\]
Clearly, $\chi$ is continuous and one-to-one. Set $Y:=\chi(X)$ and $Y_n:= \chi(X_n)$ for every $n\in\NN$.

We claim that $\SP(Y) \cap \SP_n(\II) \subseteq \SP_n(Y_n)$ for every $n\in\NN$. Indeed, fix $t\in \SP(Y) \cap \SP_n(\II)$.  So there are distinct $x_1,\dots,x_n\in \II$, distinct $y_1,\dots,y_m \in Y$, nonzero real numbers $a_1,\dots,a_m$ and nonzero numbers $\lambda_1,\dots, \lambda_n \in [-n,n]$ such that
\begin{equation} \label{equ:Free-1}
t  =a_1 y_{1} +\cdots + a_m y_{m} =\lambda_1 x_1 +\cdots +\lambda_n x_n.
\end{equation}
For every $1\leq i\leq m$ take $x_i\in X$ and $n_i\in \NN$ such that
\[
y_i=\chi(x_i) \mbox{ and } x_i\in I_{n_i}, \mbox{ so } y_i=\sum_{j\leq A_{n_i}} g_{j,n_i}(x_i).
\]
Taking into account that all $g_{j,n_i}(x_i)$ are distinct elements of the basis $\II$ of $\VV(\II)$, equality (\ref{equ:Free-1}) implies that $m\leq n$ and $\{ a_1,\dots,a_m\} \subseteq \{ \lambda_1,\dots, \lambda_n\}$. Thus $t\in \SP_n(Y_n)$.

For every $n\in\NN$, take the maximal $l(n)\in \NN$ such that $S_{l(n)} \leq n$ (note that $S_1=1$). Then the same argument as in the claim shows that $Y\cap \SP_n(\II)=Y_{S_{l(n)}}$, for every $n\in\NN$. Fix a closed subset $F$ of $X$. Then for every $n\in\NN$ we have
\[
\chi(F) \cap \SP_n(\II)=\chi(F)\cap Y\cap \SP_n(\II)=\chi(F) \cap Y_{S_{l(n)}}=\chi(F)\cap \chi\big( X_{S_{l(n)}}\big)=\chi\big( F\cap X_{S_{l(n)}}\big)
\]
is a closed subset of $\SP_n(\II)$. Since $\VV(\II)=\bigcup_n \SP_n(\II)$ is a $k_\omega$-space by Theorem \ref{t:Free-k-space}, we obtain that $\chi(F) $ is closed in $\VV(\II)$. Therefore $\chi$ is a closed map. Thus $\chi$ is a homeomorphism of $X$ onto $Y$.
Finally, Proposition \ref{p:Free-MMO-2} implies that $\SP(Y)$ is a closed subspace of $\VV(\II)$ and is topologically isomorphic to $\VV(X)$.
\end{proof}


\section{Topological properties of free topological vector spaces}



For  a subset $A$ of a tvs $E$, we denote the convex hull of $A$ by  $\conv(A)$, so
\[
\conv(A) :=\left\{ \lambda_1 a_1 +\cdots+ \lambda_n a_n : \; \lambda_1,\dots,\lambda_n \geq 0, \, \sum_{i=1}^n \lambda_i =1, \, a_1,\dots,a_n\in A, \, n\in\NN\right\}.
\]
Denote by $\mathbf{LCS}$ the category of all locally convex spaces and continuous linear operators. Let $(E,\tau)$ be a topological vector space and let $\Nn(E)$ be a base of neighborhoods at zero of $E$. Then the family $\widehat{\Nn}(E):=\big\{ \conv(U) : U\in \Nn(E)\big\}$ forms a locally convex vector topology $\widehat{\tau}$ on $E$. So $\widehat{\tau}$ is the strongest locally convex vector topology on $E$ which is coarser than the origin topology $\tau$. The lcs $(E,\widehat{\tau})$ is called the {\em locally convex modification} of $E$. Let $E$ and $H$ be topological vector spaces and $T:E\to H$ a continuous linear operator. Define the functor $\mathcal{L}:\mathbf{TVS}\to\mathbf{LCS}$ by the assignment
\[
(E,\tau) \to \mathcal{L}(E):=(E,\widehat{\tau}), \quad \mathcal{L}(T):= T.
\]

For a Tychonoff space $X$, the underlying group of $A(X)$ we denote by $A_a(X)$, and  the underlying vector spaces of $L(X)$ and $\VV(X)$ are denoted by $L_a(X)$ and $\VV_a(X)$, respectively. Below we obtain some relations between  $\VV(X)$, $L(X)$ and $A(X)$.
\begin{proposition} \label{p:Free-tvs-lcs}
Let $X$ be a Tychonoff space. Then
\begin{enumerate}
\item[{\rm (i)}] $\mathcal{L}\big(\VV(X)\big)=L(X)$; 
\item[{\rm (ii)}] the identity map $id_X:X\to X$ extends to a canonical homomorphisms $id_{A(X)}:A(X)\to \VV(X)$, which is an embedding of topological groups;
\item[{\rm (iii)}] $id_{A(X)}\big(A(X)\big)$ is closed in $\VV(X)$.
\end{enumerate}
\end{proposition}

\begin{proof}
(i) follows from the definitions of free lcs and free tvs.

(ii)  As $\pmb{\nu}_X \leq \pmb{\mu}_X $, we obtain $\pmb{\nu}_X |_{A_a(X)} \leq \pmb{\mu}_X|_{A_a(X)}$. On the other hand, by the definition of $A(X)$, there is a continuous homomorphism from $A(X)$ into  $\VV(X)$, and since the topology $\tau_{A(X)}$ of $A(X)$ is $\pmb{\nu}_X |_{A_a(X)}$ by \cite{Tkac}, we obtain that $\pmb{\nu}_X |_{A_a(X)} \geq \pmb{\mu}_X|_{A_a(X)}$. Therefore $\pmb{\nu}_X |_{A_a(X)} = \pmb{\mu}_X|_{A_a(X)}=\tau_{A(X)}$. Thus $id_{A(X)}$ is an embedding.

(iii) By \cite{Tkac},  $A(X)$ is closed in the topology $\pmb{\nu}_X$. Thus $A(X)$ is also closed in $\pmb{\mu}_X$.
\end{proof}

Proposition \ref{p:Free-tvs-lcs} allows us  to reduce easily the study of some topological properties of $L(X)$ and $\VV(X)$ to the study of the corresponding properties for $A(X)$. We demonstrate below such a reduction.

It is known that $A(X)$ is Lindel\"{o}f if and only if $X^n$ is Lindel\"{o}f for every $n\in\NN$, see Corollary 7.1.18 in \cite{ArT}. An analogous result holds for $L(X)$ and $\VV(X)$.
\begin{proposition} \label{p:Free-Lindelof}
Let $X$ be a Tychonoff space. Then
\begin{enumerate}
\item[{\rm (i)}] $L(X)$ is Lindel\"{o}f if and only if $X^n$ is Lindel\"{o}f for every $n\in\NN$;
\item[{\rm (ii)}] $\VV(X)$ is Lindel\"{o}f if and only if $X^n$ is Lindel\"{o}f for every $n\in\NN$.
\end{enumerate}
\end{proposition}
\begin{proof}
We prove only (ii) as (i) is proved in an analgous manner. Assume that $\VV(X)$ is a Lindel\"{o}f space.  By Proposition \ref{p:Free-tvs-lcs}, $A(X)$ is a closed subspace of $L(X)$. Hence $A(X)$ is also Lindel\"{o}f. Therefore $X^n$ is Lindel\"{o}f for every $n\in\NN$ by Corollary 7.1.18 of \cite{ArT}.

Conversely, let $X^n$ be a Lindel\"{o}f space for every $n\in\NN$. Then the disjoint sum
\[
Y:= \bigsqcup_{n\in \NN} Y_n, \mbox{ where } Y_n:= [-n,n]^n \times X^n,
\]
is also  a Lindel\"{o}f space, see Corollary 3.8.10 in \cite{Eng}. Consider the map $T:Y\to \VV(X)$ defined by
\[
T(y):= T_n(y) \mbox{ if } y=(a_1,\dots,a_n)\times(x_1,\dots,x_n) \in Y_n \mbox{ and } T_n \big( y\big) := a_1 x_1 +\cdots +a_n x_n.
\]
Clearly, the map $T$ is continuous. Thus $\VV(X)$ is a Lindel\"{o}f space.
\end{proof}

Bel'nov \cite{Belnov} proved that if a topological group $G$ is algebraically generated by a Lindel\"{o}f subspace, then $G$ is topologically isomorphic to a subgroup of the product of some family of second-countable groups. As $\VV(X)$ and $L(X)$ are algebraically generated by the continuous image of $[-1,1]\times X$ and a product of a compact space and a Lindel\"{o}f space is also Lindel\"{o}f, we obtain

\begin{proposition} \label{p:Free-Belnov}
Let $X$ be a Lindel\"{o}f space. Then $\VV(X)$ and $L(X)$ are  topologically isomorphic to a subgroup of the product of some family of second-countable groups. \end{proposition}

Recall that a subset $A$ of a topological space $X$ is called {\em functionally bounded} if every continuous real-valued function $f\in C(X)$ is bounded on $A$. We shall use the following result, see Lemma  10.11.3 in \cite{Banakh-Survey}.

\begin{fact}[\cite{Banakh-Survey}] \label{p:Free-Banakh}
Let $X$ be a Tychonoff space and $A$ be a functionally bounded subset of $L(X)$. Then the set $\bigcup_{v\in A} \supp(v)$ is functionally bounded in $X$ and there is an $n\in\NN$ such that $A\subseteq \SP_n(X)$.
\end{fact}
An analogous result holds also for $\VV(X)$.

\begin{proposition} \label{p:Free-V(X)-bounded}
Let $X$ be a Tychonoff space and $A$ be a functionally bounded subset of $\VV(X)$. Then the set $\bigcup_{v\in A} \supp(v)$ is functionally bounded in $X$ and there is $n\in\NN$ such that $A\subseteq \SP_n(X)$.
\end{proposition}
\begin{proof}
Proposition \ref{p:Free-tvs-lcs} implies that the identity map $id: \VV(X)\to L(X)$ is continuous. So $A$ is also a functionally bounded subset of $L(X)$ and Fact  \ref{p:Free-Banakh} applies.
\end{proof}
Two Tychonoff spaces $X$ and $Y$ are called {\em $\VV$-equivalent} if the free topological vector spaces $\VV(X)$ and $\VV(Y)$ are isomorphic as topological vector spaces. Analogously, $X$ and $Y$ are said to be {\em $L$-equivalent} if the free locally convex spaces $L(X)$ and $L(Y)$ are isomorphic as topological vector spaces. A topological property $\mathcal{P}$ is called {\em $\VV$-invariant} ({\em $L$-invariant}) if every space $Y$ which is  $\VV$-equivalent (respectively, $L$-equivalent) to a Tychonoff space $X$ with $\mathcal{P}$ also has the property $\mathcal{P}$.

A subspace $Y$ of $\VV(X)$ ($L(X)$) is called a {\em topological basis} of $\VV(X)$ (respectively, $L(X)$) if $Y$ is a vector basis of $\VV(X)$ ($L(X)$) and the maximal vector space topology (maximal locally convex space topology) on the abstract vector space $\VV_X$ which induces on $Y$ its original topology coincides with the topology of $\VV(X)$ (respectively, $L(X)$).
The next theorem has a similar proof to the proof of Theorem 7.10.10 of \cite{ArT}.

\begin{theorem} \label{t:Free-pseudocompactness}
Pseudocompactness is a $\VV$-invariant property and an $L$-invariant property.
\end{theorem}
\begin{proof}
We prove the theorem only for the case $\VV(X)$. Let $X$ and $Y$ be $\VV$-equivalent  spaces. The theorem is clear if $X$ or $Y$ is finite, so we assume that $X$ and $Y$ are infinite. Assume that $Y$ is pseudocompact. So we can assume that $Y$ is a topological basis of the space $\VV(X)$. Note that $Y$ is a closed subspace of $\VV(X)$ since $Y$ is closed in $\VV(Y)$. Suppose for a contradiction that $X$ is not pseudocompact. Then $X$ contains a discrete family $\UU=\{ U_n : n\geq 0\}$ of non-empty open sets. For every $n\in \NN$, choose a point $x_n\in U_n$.

Define the function $d(x,v): X\times \VV(X) \to \IR$ as follows: if $v=\lambda_1 x_1 +\cdots + \lambda_k x_k \in\VV(X)$, where all $x_i$ are distinct and $\lambda_i \not= 0$ for every $i=1,\dots,k$, then $d(x,v)=\lambda_i$ if $x=x_i$ for some $1\leq i\leq k$, and $d(x,v)=0$ otherwise.


Since $Y$ is a vector basis of $\VV(X)$, for every $n\in\NN$ there exists $y_n\in Y$ such that $d(x_n,y_n)\not= 0$. So, for every $n\in\NN$, we have
\begin{equation} \label{equ:Free-ps-0}
y_n = \lambda_{0,n} x_n +\sum_{j=1}^{k_n} \lambda_{j,n} t_{j,n}, \mbox{ where } \lambda_{j,n}\not= 0 \mbox{ for every } 0\leq j\leq k_n,
\end{equation}
and all letters $t_{1,n},\dots,t_{k_n,n}\in X\setminus\{ x_n\}$ are distinct (possibly, $k_n =0$). Passing to a subsequence of $\{ y_n: n\in\NN\}$ if it is needed, we can assume that $d(x_j,y_n)= 0$ whenever $n<j$.

By induction on $n\geq 0$, we define continuous real-valued functions $f_n$ on $X$ as follows.  Set $f_0\equiv 0$. Assume that we have defined $f_0,\dots,f_{n-1}$. Put $g_n :=\sum_{i=0}^{n-1} f_i$. Then there exists a continuous real-valued  function $f_n$ on $X$ such that  $f_{n}(X\setminus U_{n})=\{ 0\}$ and $f_{n}(t_{j,i})=0$ at each point $t_{j,i}$ with $i\leq n$ that belongs to $U_n$ and also satisfies
\begin{equation} \label{equ:Free-ps-1}
f_{n} (x_{n})= \frac{n}{\lambda_{0,n}} + \frac{1}{\lambda_{0,n}} \sum_{j\in A_n} \left| \lambda_{j,n} g_n \big(t_{j,n}\big) \right|,
\end{equation}
where $A_n$ is the set of all $1\leq j\leq k_n$ such that $t_{j,n} \in U_0\cup\cdots\cup U_{n-1}$.

Since the family $\UU$ is discrete, the function $f:=\sum_{n\in\NN} f_n$ is continuous on $X$. Moreover,
\begin{equation} \label{equ:Free-ps-3}
f=f_n \mbox{ on } U_n, \mbox{ and } f=g_n  \mbox{ on } U_0\cup\cdots\cup U_{n-1}.
\end{equation}
In addition, the definition of $f$ implies that for all $n\in\NN$ and all $1\leq j\leq k_n$,
\begin{equation} \label{equ:Free-ps-2}
f(t_{j,n})=0 \quad \mbox{ whenever } j\not\in A_n.
\end{equation}
Extend $f$ to a continuous functional $\widetilde{f}: \VV(X)\to\IR$. From (\ref{equ:Free-ps-0})  it follows that
\[
\begin{split}
\widetilde{f}(y_n) & = \lambda_{0,n} f(x_n) + \sum_{j=1}^{k_n} \lambda_{j,n} f(t_{j,n})  \quad \mbox{ (by (\ref{equ:Free-ps-2}))} \\
 & =\lambda_{0,n} f(x_n) + \sum_{j\in A_n} \lambda_{j,n} f(t_{j,n}) \quad \mbox{ (by (\ref{equ:Free-ps-1}) and (\ref{equ:Free-ps-3}))} \\
 & = n  + \sum_{j\in A_n} \left| \lambda_{j,n} g_n \big(t_{j,n}\big) \right| + \sum_{j\in A_n} \lambda_{j,n} g_n \big(t_{j,n}\big) \geq n
\end{split}
\]
for every $n\in\NN$. As $\{ y_n: n\in\NN\} \subset Y$, we conclude that $Y$ is not pseudocompact. This contradiction completes the proof.
\end{proof}

Graev proved in \cite{Gra} that compactness is an $A$-invariant property. Below we prove an analogous result with a similar proof.
\begin{theorem} \label{t:Free-compactness-V-L}
Let $X$ and $Y$ be any $\VV$-equivalent spaces or $L$-equivalent spaces. Then
\begin{enumerate}
\item[{\rm (i)}] if $X$ is compact, then so is $Y$;
\item[{\rm (ii)}]  if $X$ is compact and metrizable, then so is $Y$.
\end{enumerate}
\end{theorem}
\begin{proof}
We prove the theorem only for $\VV$-equivalent spaces.

(i) Since $Y$ is closed in $\VV(Y)$ it is also closed in $\VV(X)$. Theorem \ref{t:Free-pseudocompactness} implies that $Y$ is pseudocompact. So, fo some $n\in\NN$,  $Y$ is a closed subset of the compact set $\SP_n(X)$ by Proposition \ref{p:Free-V(X)-bounded}. Thus $Y$ is compact.

(ii) By the proof of (i), $Y$ is a closed subset of $\SP_n(X)$ for some $n\in\NN$. Since $\SP_n(X)$ is a continuous image of the compact metrizable space $[-n,n]^n \times X^n$ we obtain that $\SP_n(X)$ is a compact metrizable space. Thus $Y$ is  compact and metrizable.
\end{proof}

\begin{proposition} \label{p:Free-normal}
If $X$ is a Tychonoff non-normal space, then $L(X)$ and $\VV(X)$ are also not normal spaces.
\end{proposition}
\begin{proof}
The proposition follows from the fact that $X$ is a closed subspace of $L(X)$ and $\VV(X)$ and the fact that a closed subspace of a normal space is also normal.
\end{proof}
\begin{remark} {\em
We note that even if $X$ is a normal space, $\VV(X)$ and $L(X)$ need not be normal spaces. This follows from the following three facts: (1) $A(X)$ is a closed subgroup of  $\VV(X)$ and $L(X)$, (2) $A(X)$ contains a closed homeomorphic copy of $X^n$ for every $n\in\NN$, see Corollary 7.1.16 of \cite{ArT}, and (3) the square of a normal space can be not normal, see Example 2.3.12 of \cite{Eng}.}
\end{remark}

For a tvs $E$, we denote by $E'$ the topological dual space of $E$.
\begin{proposition} \label{p:Free-compatible}
For every Tychonoff space $X$, the topologies $\pmb{\nu}_X$ and $\pmb{\mu}_X$ are compatible, i.e. $\VV(X)$ and $L(X)$ have the same continuous functionals.
\end{proposition}
\begin{proof}
The topology $\tau$ on $\VV_X$ defined by $\VV(X)'$ is locally convex. So, by the definition of $\pmb{\nu}_X$, we have $\tau\leq \pmb{\nu}_X$. Thus any $\chi\in \VV(X)'$ is also continuous in $\pmb{\nu}_X$, and hence $\VV(X)' \subseteq L(X)'$. The converse assertion is trivial.
\end{proof}

Using Fact \ref{f:Free-Completness} we obtain
\begin{proposition}
If $X$ is a Tychonoff space such that $\VV(X)$ is complete, then $X$  is Dieudonn\'{e} complete.
\end{proposition}
\begin{proof}
Note that $A(X)$  is a closed subspace of $\VV(X)$ by Theorem \ref{t:Free-exists}. So, if $\VV(X)$ is complete, then $A(X)$ is also complete. Thus $X$ is Dieudonn\'{e} complete by Fact \ref{f:Free-Completness}(i).
\end{proof}
We do not know whether the converse is true

\begin{question} \label{q:Free-Comleteness}
Let $X$ be a Dieudonn\'{e} complete  Tychonoff  space. Is $\VV(X)$ complete?
\end{question}
Note that Protasov \cite{Prot} proved the completeness of $\VV(\kappa)$ for any cardinal $\kappa$.

A natural question arises: {\em For which Tychonoff spaces is $X$ the space $\VV(X)$ locally convex}? Note that $\VV(X)$ is locally convex if and only if $\VV(X)=L(X)$. Below we obtain a  complete answer to this question in the case that  $X$ is a $k$-space. We start with a necessary condition for the equality $\pmb{\mu}_X =\pmb{\nu}_X$.

\begin{proposition} \label{p:Free-Compact}
If $\VV(X)$ is locally convex, then $X$ does not contain infinite compact subsets.
\end{proposition}
\begin{proof}
Suppose that $X$ contains an infinite compact subset  $K$. By Proposition \ref{p:Free-Subspace},  $\VV(K)$ is also locally convex. So $\VV(K)=L(K)$ and hence $L(K)$ is complete by Theorem \ref{t:Free-k-space}.  Since $K$ is also Dieudonn\'{e} complete, we obtain a contradiction with Fact \ref{f:Free-Completness}(ii).
\end{proof}

\begin{remark} {\em
It is known that the family $\mathcal{S}$ of all seminorms on $\VV_a(X)$ which are continuous on $X$ defines a free locally convex vector topology $\pmb{\nu}_X$ on $\VV_a(X)$. So the family $\mathcal{S}$ defines the topology  $\pmb{\mu}_X$ of $\VV(X)$ if and only if $\VV(X)=L(X)$. }
\end{remark}

Any topological vector space which is a reflexive abelian group is locally convex, see \cite{Ban}. This remark and Propositions \ref{p:Free-compatible} and \ref{p:Free-Compact} imply
\begin{corollary}
If $\VV(X)$ is a reflexive group, then $\VV(X)=L(X)$. In particular, $X$ does not contain infinite compact sets.
\end{corollary}

We do not know  whether the converse is true:
\begin{question} \label{q:Free-Compact}
If $X$ does not contain infinite compact subsets,  is $\VV(X)$  locally convex?
\end{question}

\begin{theorem} \label{t:Free-lcs}
If $X$ is a $k$-space, then $\VV(X)$ is locally convex if and only if $X$ is a discrete countable space, that is $\VV(X)$ equals $\phi$ or $\IR^n$ for some $n\in\NN$.
\end{theorem}

\begin{proof}
Let $X$ be a $k$-space such that $\VV(X)$ is locally convex. By Proposition \ref{p:Free-Compact}, $X$ does not contain infinite compact subsets. So being a $k$-space, the space $X$ is discrete. I.~Protasov \cite{Prot} (see also \cite{Gabr}) proved that $\VV(D)$ is not locally convex for every uncountable discrete space $D$. Thus $X$ must be countable. Conversely, if $X$ is a discrete countable space, then $\VV(X)=L(X)$ by Proposition 4.1.4 of \cite{Jar}.
\end{proof}

Proposition \ref{p:Free-Compact} and Theorem \ref{t:Free-lcs} motivate the following question: {\em When does $\VV(X)$ contain an infinite-dimensional locally convex subspaces (for example $L(Y)$ for some infinite $Y$)}?

It is well known that the free groups $F(X)$ and $A(X)$ are  Fr\'{e}chet--Urysohn spaces if and only if $X$ is a discrete space (see \cite{OrdThomas}), and the space $L(X)$ is a $k$-space if and only if $X$ is a discrete countable space by Fact \ref{f:Free-L}. In Theorem \ref{t:Free-Frechet} below we prove an analogous result for $\VV(X)$. We  need the following result.

\begin{fact}[\cite{Gabr}] \label{f:Free-tightness}
If $X$ is an uncountable Tychonoff space, then $\VV(X)$ has uncountable tightness  and is not a $k$-space.
\end{fact}

\begin{proposition} \label{p:Free-Ord-Smith}
Let $\VV(X)$ be a sequential space.
 \begin{enumerate}
\item[{\rm (i)}] If $X$ is non-discrete, then $\VV(X)$ has sequential order $\omega_1$.
\item[{\rm (ii)}] If $X$ is discrete, then $X$ is countable; so $\VV(X)$ equals $\phi$ or $\IR^n$ for some $n\in\NN$.
\end{enumerate}
\end{proposition}
\begin{proof}
(i) Since $X$ is a closed subspace of the sequential space $\VV(X)$ by Theorem \ref{t:Free-exists}, $X$ is also sequential. Being non-discrete $X$ contains a non-trivial convergent sequence $\mathfrak{s}$. Proposition \ref{p:Free-Subspace} implies that $\VV(\mathfrak{s})$ is isomorphic to a closed subspace of $\VV(X)$. So $\VV(\mathfrak{s})$ is a sequential space. As $A(\mathfrak{s})$ has sequential order $\omega_1$ by \cite[Theorem 3.9]{OrdThomas} (see also \cite[Theorem 2.3.10]{ZP2}), we obtain that also $\VV(X)$ has sequential order $\omega_1$.

(ii) follows from Fact \ref{f:Free-tightness}.
\end{proof}

It is well known that $\phi$ is a sequential non-Fr\'{e}chet--Urysohn space.
\begin{theorem} \label{t:Free-Frechet}
For a Tychonoff space $X$, the space $\VV(X)$ is Fr\'{e}chet--Urysohn if and only if $X$ is finite. In particular, $\VV(X)$ is metrizable if and only if $X$ is finite.
\end{theorem}

\begin{proof}
Assume that $\VV(X)$ is a Fr\'{e}chet--Urysohn space. Then $X$ is discrete by Proposition \ref{p:Free-Ord-Smith}. If $X$ is infinite, then $\VV(X)$ contains $\phi$ as a direct summand. So $\VV(X)$ is not Fr\'{e}chet--Urysohn, a contradiction. Thus $X$ must be finite.
Conversely, if $X$ is finite, then $\VV(X)=\IR^{|X|}$ is a locally compact metrizable space.
\end{proof}

\begin{remark} {\em
Note that $\VV(X)$ is a locally compact space if and only if $X$ is finite. This follows from the  Principal Structure Theorem for locally compact abelian groups (Theorem 25 of \cite{Morris1}) observing that $\VV(X)$ is a torsion-free divisible abelian group and so $\VV(X)=\IR^n$ for some  $n\in \NN$. }
\end{remark}

Recall (see \cite{Mich}) that a topological space $X$ is called \emph{cosmic}, if $X$ is a regular space with a countable network (a family $\mathcal{N}$ of subsets of $X$ is called a \emph{network} in $X$ if, whenever $x\in U$ with $U$ open in $X$, then $x\in N\subseteq U$ for some $N\in\mathcal{N}$). Michael proved in \cite{Mich} that a regular space $X$ is cosmic if and only if $X$ is a continuous image of a separable metric space.
\begin{proposition} \label{p:Free-cosmic}
Let $X$ be a Tychonoff space. Then $\VV(X)$ is a cosmic space if and only if $X$ is cosmic.
\end{proposition}
\begin{proof}
If $\VV(X)$ is cosmic, then $X$  is also cosmic as a subspace of a cosmic space.

Assume that $X$ is a cosmic space. So there is a separable metric space $M$ and a continuous surjective map $f: M\to X$, see \cite{Mich}. For every $n\in\NN$ set
\[
Y_n := [-n,n]^n \times M^n ,
\]
and define the map $T_n :Y_n \to \SP_n(X)$ by
\[
T_n (a_1,\dots,a_n,y_1,\dots,y_n):= a_1 f(y_1)+\cdots +a_n f(y_n), \quad a_i\in[-n,n], y_i\in M, i=1,\dots,n.
\]
Clearly, $Y_n$ is a  separable metric space and $T_n$ is continuous and onto.

Set $Y:= \bigsqcup_{n\in\NN} Y_n$ and define the map $T:Y\to \VV(X)$ by
\[
T(y):= T_n (y), \quad \mbox{ if } \; y\in Y_n.
\]
Clearly, $Y$ is a  separable metric space and $T$ is continuous and onto. Thus $\VV(X)$ is a cosmic space by \cite{Mich}.
\end{proof}

\begin{corollary}
For a Tychonoff space $X$ the following assertions are equivalent:
\begin{enumerate}
\item[{\rm (i)}] $X$ is a cosmic space;
\item[{\rm (ii)}] $A(X)$ is  a cosmic space;
\item[{\rm (iii)}] $L(X)$ is  a cosmic space.
\end{enumerate}
\end{corollary}
\begin{proof}
(i)$\Rightarrow$(iii) follows from  Proposition \ref{p:Free-cosmic} since the topology $\pmb{\nu}_X$ of $L(X)$ is weaker than the topology $\pmb{\mu}_X$ of $\VV(X)$: if $\VV(X)$ is the image of a separable metric space under a continuous map, then so is $L(X)$.

(iii)$\Rightarrow$(ii) and (ii)$\Rightarrow$(i) follow from the facts that $A(X)$ is a subspace of $L(X)$ and $X$ is a subspace of $A(X)$.
\end{proof}

Recall that a space $X$ has {\em countable tightness} if whenever $x\in \overline{A}$ and $A\subseteq X$, then $x\in \overline{B}$ for some countable $B\subseteq A$.
We use the following remarkable result of Arhangel'skii, Okunev and Pestov which shows that the topology of $A(X)$ is rather complicated and unpleasant even for the simplest case of a metrizable space $X$. Denote by $X'$ the set  of all non-isolated points in a space $X$.
\begin{fact}[\cite{AOP}] \label{tAOP}
Let $X$ be a metrizable space. Then
\begin{enumerate}
\item[{\rm (i)}] the tightness of $A(X)$ is countable if and only if the set $X'$ is separable;
\item[{\rm (ii)}] $A(X)$ is a $k$-space if and only if $X$ is locally compact and the set $X'$  is separable.
\end{enumerate}
\end{fact}

For a metric space $X$, the space $L(X)$ has countable tightness if and only if $X$ is separable, see \cite{Gab-MSJ}. The same holds also for $\VV(X)$. 

\begin{theorem} \label{t:Free-tight}
Let $X$ be a metrizable space. Then $\VV(X)$ has countable tightness if and only if $X$ is separable.
\end{theorem}
\begin{proof}
Assume that  $\VV(X)$ has countable tightness. Then $A(X)$ has countable tightness as a subspace of $\VV(X)$, see Proposition \ref{p:Free-tvs-lcs}. So $X'$ is separable by Fact \ref{tAOP}. To prove that $X$ is separable we have to show that the set $D:= X\setminus X'$ is countable.

Suppose for a contradiction that $D$ is uncountable. Then there is a positive number $c$ and an uncountable subset $D_0$ of $D$ such that $B_c(d)=\{ d\}$ for every $d\in D_0$, where $B_c(d)$ is the $c$-ball centered at $d$. It is easy to see that $D_0$ is a clopen  subset of $X$. So $X=X_0 \sqcup D_0$, where $X_0:= X\setminus D_0$. Now Corollary \ref{c:Free-retract} implies that $\VV(X)=\VV(X_0) \oplus \VV(D_0)$. Hence $\VV(D_0)$ also has countable tightness, but this contradicts  Fact \ref{f:Free-tightness}. Thus $D$ is countable, and hence $X$ is separable.

Conversely, if $X$ is separable it is a cosmic space. So $\VV(X)$ is a cosmic space by Proposition \ref{p:Free-cosmic}. Thus $\VV(X)$ has countable tightness, see \cite{Mich}.
\end{proof}

By Fact \ref{f:Free-L},  $L(X)$ is a $k$-space if and only if $X$ is a countable discrete space. For the case $\VV(X)$ the situation is more complicated.
\begin{theorem} \label{t:Free-k-space-metr}
For an infinite metrizable space $X$ the following assertions are equivalent:
\begin{enumerate}
\item[{\rm (i)}] $X$ is a locally compact separable (metric) space;
\item[{\rm (ii)}] $\VV(X)$ is a non-Fr\'{e}chet--Urysohn sequential space;
\item[{\rm (iii)}] $\VV(X)$ is a $k$-space;
\item[{\rm (iv)}] $\VV(X)$ is a $k_\omega$-space.
\end{enumerate}
\end{theorem}

\begin{proof}
(i)$\Rightarrow$(vi). Since $X$ is locally compact metrizable and separable, $X$ is a $k_\omega$-space. So $\VV(X)$ is  a $k_\omega$-space by Theorem \ref{t:Free-k-space}.

  (iv)$\Rightarrow$(iii) is trivial.

(iii)$\Rightarrow$(i). If $ \VV(X)$ is a $k$-space, then $A(X)$ is a $k$-space as a closed subspace of $\VV(X)$. Then $X$ is locally compact and the set $X'$  is separable by Fact \ref{tAOP}. To prove that $X$ is separable we have to show that the set $D:= X\setminus X'$ is countable. We repeat our argument from the proof of Theorem \ref{t:Free-tight}.

Suppose for a contradiction that $D$ is uncountable. Then there is a positive number $c$ and an uncountable subset $D_0$ of $D$ such that $B_c(d)=\{ d\}$ for every $d\in D_0$, where $B_c(d)$ is the $c$-ball centered at $d$. It is easy to see that $D_0$ is a clopen  subset of $X$. So $X=X_0 \sqcup D_0$, where $X_0:= X\setminus D_0$. Now Corollary \ref{c:Free-retract} implies that $\VV(X)=\VV(X_0) \oplus \VV(D_0)$. Hence $\VV(D_0)$ also  is a $k$-space, but this contradicts Fact \ref{f:Free-tightness}. Thus $D$ is countable, and hence $X$ is separable.

(ii)$\Rightarrow$(iii) is trivial.

Let us prove  (iv)$\Rightarrow$(ii). Since $\VV(X)$ is a $k_\omega$-space, the closed subset $X$ is also a $k_\omega$-space. Let $X=\bigcup_{n\in\NN} C_n $ be a $k_\omega$-decomposition of $X$. Then $\VV(X)=\bigcup_{n\in\NN} \SP_n(C_n)$ is a $k_\omega$-decomposition of $\VV(X)$, see Theorem \ref{t:Free-k-space}. As each $\SP_n(C_n)$ is a metrizable compactum, the space $\VV(X)$ is sequential. Since $X$ is infinite, Theorem \ref{t:Free-Frechet} implies that $\VV(X)$ is a non-Fr\'{e}chet--Urysohn space.
\end{proof}


\section{Vector space properties of free topological vector spaces}


First we note that any topological vector space is a quotient space of a free topological vector space.
\begin{proposition} \label{p:Free-quot-tvs}
Let $E$ be any topological vector space and $\VV(E)$ the free topological vector space on $E$. Then the canonical continuous linear map of $\VV(E)$ onto $E$ is a quotient map.
\end{proposition}
\begin{proof}
Denote by $\tau$ the topology of $E$. The identity map $\phi:E\to E$ can be extended to a continuous linear map $\Phi:\VV(E)\to E$. If $\Phi$ is not a quotient map, then  the quotient vector space topology $\tau_1$ on the underlying vector space $E_a$ of $E$ is strictly  finer than $\tau$. Therefore $\phi$ is not continuous, a contradiction. Thus $\tau_1=\tau$.
\end{proof}

Recall that  a topological vector space $E$ is called {\em barrelled} if every barrel in $E$ is a neighborhood of zero.
Following Saxon \cite{Sa}, a topological vector space $E$ is called {\em Baire-like} if every increasing sequence $\{ A_{n}\}_{n\in\NN}$  of absolutely convex closed subsets covering $E$ contains a member which is a neighborhood of zero. Clearly,   Baire lcs $\Rightarrow$ Baire-like.

\begin{theorem} \label{t:Free-barrel}
Let $X$ be a Tychonoff space. Then
\begin{enumerate}
\item[{\rm (i)}]  $\VV(X)$ is barrelled if and only if $X$ is discrete.
\item[{\rm (ii)}] Let  $X$ be discrete. Then $\VV(X)$ is a Baire-like space if and only if $X$ is finite.
\item[{\rm (iii)}]  Let  $X$ be discrete. Then $\VV(X)$ is a Baire space if and only if $X$ is finite.
\end{enumerate}
\end{theorem}

\begin{proof}
(i) Assume that $\VV(X)$ is a barrelled tvs. Suppose for a contradiction that $X$ is not discrete. Let $x_0\in X$ be a non-isolated point and take a net $N=\{ x_i\}_{i\in I}$ in $X$ converging to $x_0$ (we assume that $x_0\not\in N$). Set
\[
A:=\bigcup \left\{ \left[ -\frac{1}{2},\frac{1}{2}\right]x_i : \; i\in I \right\} \cup \bigcup \big\{ [-1,1]x :\: x\in X\setminus N \big\} \subset\VV_X,
\]
and let $B$ be the absolute convex hull of $A$. Then $B$ is a barrel in $\VV(X)$ and $B=\conv(A)$. Note that, for every $i\in I$,
\begin{equation} \label{equ:Free-bar}
\lambda x_i + \mu x_0 \in B \; \mbox{  if and only if } \; 2\lambda,\mu \in [-1,1].
\end{equation}

We show that $B$ is not a neighborhood of zero in $\VV(X)$. Indeed, otherwise we can find a neighborhood $U$ of $x_0$ such that $x-x_0 \in B $ for every $x\in U$. So, for every $j\in I$ such that $x_j\in U$, we obtain $x_j - x_0 \in B$ that contradicts (\ref{equ:Free-bar}). Thus $X$ is discrete.

Conversely, let $X$ be a discrete space.
We shall use the following simple description of the topology $\pmb{\mu}_X$ of $\VV(X)$ given in the proof of Theorem 1 in \cite{Prot}.  For each $x\in X$, choose some $\lambda_x  >0$, and denote by $\mathcal{S}_X$ the family of all subsets of $\VV_X$ of the form
\[
\bigcup \big\{ [-\lambda_x,\lambda_x ] x : x\in X \big\}.
\]
For every sequence $\{ S_k\}_{k\geq 0}$ in $\mathcal{S}_X$, we put
\[
\sum_{k\geq 0} S_k := \bigcup_{k\geq 0} (S_0 + S_1 +\cdots + S_k),
\]
and denote by $\mathcal{N}_X$ the family of all subsets of $\VV_X$ of the form $\sum_{k\geq 0} S_k$. Then $\mathcal{N}_X$ is a basis at zero, $\mathbf{0}$, for $\pmb{\mu}_X$.

Now let $B$ be a barrel in $\VV(X)$. For every $x\in X$ choose $\lambda_x >0$ such that $[-\lambda_x,\lambda_x] x \subseteq B$ and set
\[
B_0 :=\conv\left\{ \bigcup \{ [-\lambda_x,\lambda_x] x : x\in X \} \right\}.
\]
Then $B_0$ is a barrel in $\VV(X)$ and $B_0 \subseteq B$. For every integer $k\geq 0$,  set
\[
S_k := \bigcup \left\{ \left[ -\frac{\lambda_x}{2^{k+1}},\frac{\lambda_x}{2^{k+1}}\right]x : \; x\in X \right\}.
\]
Then the neighborhood $\sum_{k\geq 0} S_k$ of zero in $\VV(X)$ is a subset of $B_0$. Therefore $B$ is also a neighborhood of zero in $\VV(X)$. Thus $\VV(X)$ is a barrelled space.

(ii) If $X$ is infinite and $S=\{ x_n\}_{n\in\NN}$ is a sequence in $X$ consisting of distinct elements, then
\[
\VV(X)=\bigcup_{n\in\NN} A_n, \; \mbox{ where } A_n :=  [-n,n]x_1 +\cdots+[-n,n]x_n  + \VV_{X\setminus S}.
\]
Since $X$ is discrete, $A_n$ is closed, and hence $A_n$ is a meager closed subset of $\VV(X)$ for every $n\in \NN$. Therefore $\VV(X)$ is not Baire. Thus $X$ is  finite.

If $X$ is finite, then $\VV(X)=\IR^{|X|}$ is a Baire space.

(iii) follows from (ii).
\end{proof}

We shall identify elements $\delta(x)\in L(X)$ with the Dirac measure $\delta_x$ on $X$. So for every element $\mu= a_1 x_1 +\cdots + a_n x_n$ of $L(X)$ with distinct $x_1,\dots,x_n$ we can define the norm of $\mu$ setting
\[
\|\mu\|:=|a_1|+\cdots +|a_n|.
\]
We need the following lemma whose proof  actually can be extracted from the proof of Lemma 10.11.3 of \cite{Banakh-Survey}. Recall that a subset $M$ of a topological vector space $E$ is called {\em bounded} if for every neighborhood $U$ of zero there is $\lambda >0$ such that $M\subseteq \lambda U$.
\begin{lemma} \label{l:Free-bounded-L(X)}
Let $X$ be a Tychonoff space and $M$ a bounded subset of $L(X)$. Then the set $\{ \| \mu\|: \mu\in M\}$ is bounded in $\IR$.
\end{lemma}
\begin{proof}
Let $i:X \to \beta X$ be the identity inclusion of $X$ into the Stone--\v{C}ech compactification $\beta X$ of $X$, and let ${\bar i} : L(X)\to L(\beta X)$ be an injective linear continuous extension of $i$. Since $L(\beta X)$ is a subspace of $\CC(\CC(\beta X))$ (see \cite{Flo2,Usp2}) 
and the map ${\bar i}$ preserves the norm of measures $\mu\in L(X)$, it follows that ${\bar i}(M)$ is a bounded subset of the dual Banach space $\CC(\beta X)'$ endowed with the weak* topology. Now the Banach--Steinhaus theorem \cite[3.88]{fabian-10} implies that $M$ is bounded.
\end{proof}

\begin{theorem} \label{t:Free-L(X)-quasibarrelled}
For a Tychonoff space  $X$, the following assertions are equivalent:
\begin{enumerate}
\item[{\rm (i)}] $X$ is discrete;
\item[{\rm (ii)}] $L(X)$ is barrelled;
\item[{\rm (iii)}] $L(X)$ is quasibarrelled.
\end{enumerate}
\end{theorem}

\begin{proof}
(i)$\Rightarrow$(ii) Let $B$ be a barrel in $L(X)$. Then $B$ is a neighborhood of zero in $\VV(X)$ by Theorem \ref{t:Free-barrel}. Since $\conv(B)=B$,  Proposition \ref{p:Free-tvs-lcs} implies that $B$ is a neighborhood of zero in $L(X)$. Thus $L(X)$ is barrelled.
(ii)$\Rightarrow$(iii) is clear. Let us prove (iii)$\Rightarrow$(i).

Suppose for a contradiction that $X$ is not discrete and $x_0\in X$ is a non-isolated point of $X$. As in the proof of (i) of Theorem \ref{t:Free-barrel}, take a net $N=\{ x_i\}_{i\in I}$ in $X$ converging to $x_0$ (we assume that $x_0\not\in N$) and set
\[
A:=\bigcup \left\{ \left[ -\frac{1}{2},\frac{1}{2}\right]x_i : \; i\in I \right\} \cup \bigcup \big\{ [-1,1]x :\: x\in X\setminus N \big\} \subset\VV_X,
\]
and let $B$ be the absolute convex hull of $A$. Then $B$ is a barrel in $L(X)$ which is not a neighbourhood of zero even in $\VV(X)$ by the proof of (i) of Theorem \ref{t:Free-barrel}. The construction of $B$ and Lemma \ref{l:Free-bounded-L(X)} imply that $B$ is bornivorous. Thus $L(X)$ is not quasibarrelled. This contradiction shows that $X$ is discrete.
\end{proof}
\begin{corollary}
Let $X$ be a Tychonoff space. Then $L(X)$ is a Baire space if and only if $X$ is finite.
\end{corollary}
\begin{proof}
Assume that $L(X)$ is Baire. Then $L(X)$ is a barrelled space by \cite[Theorem 11.8.6]{NaB}. Therefore $X$ is discrete by Theorem \ref{t:Free-L(X)-quasibarrelled}. Thus $X$ is finite by (the proof of) Theorem \ref{t:Free-barrel}(ii). The converse assertion is trivial.
\end{proof}



\section{Acknowledgments}

The second author gratefully acknowledges the financial support of the research Center for Advanced Studies in Mathematics of the Ben-Gurion University of the Negev.

\bibliographystyle{amsplain}

\begin{thebibliography}{10}

\bibitem{Arhangel}
A. V. Arhangel'skii, \emph{Topological function spaces}, Math. Appl. \textbf{78}, Kluwer Academic Publishers, Dordrecht, 1992.



\bibitem{AOP}
A.~V.~Arhangel'skii, O.~G.~Okunev, V.~G.~Pestov, Free topological groups over metrizable spaces, Topology Appl. \textbf{33} (1989), 63--76.


\bibitem{ArT}
A.~V.~Arhangel'skii, M.~G.~Tkachenko, \emph{Topological groups and related strutures}, Atlantis Press/World Scientific, Amsterdam-Raris, 2008.

\bibitem{Banakh-Survey}
T. Banakh, Fans and their applications in General Topology, Functional Analysis and Topological Algebra, available in arXiv:1602.04857.


\bibitem{Ban}
{W. Banaszczyk},  \emph{Additive subgroups of topological vector
spaces}, LNM 1466, Berlin-Heidelberg-New York 1991.

\bibitem{Belnov}
V. K. Bel'nov,  On dimension of free topological groups, Proc. IVth Tiraspol Symposium on General Topology, (1979), 14--15 (in Russian).


\bibitem{Eng}
R.~Engelking, \emph{ General Topology}, Heldermann Verlag, Berlin, 1989.

\bibitem{fabian-10}
M. Fabian, P. Habala, P. H\'{a}jek, V. Montesinos, J. Pelant, V. Zizler, \emph{Banach space theory. The basis for linear and nonlinear analysis}, Springer, New York, 2010.


\bibitem{Flo1}
J.~Flood, Free topological vector spaces, Ph. D. thesis, Australian National University, Canberra, 109 pp., 1975.

\bibitem{Flo2}
J.~Flood, Free locally convex spaces, Dissertationes Math CCXXI, PWN, Warczawa, 1984.


\bibitem{Gabr}
S.~Gabriyelyan, The $k$-space property for  free locally convex spaces, Canadian Math. Bull. \textbf{57} (2014), 803--809.

\bibitem{Gab-MSJ}
S.~Gabriyelyan, A characterization of  free locally convex spaces over metrizable spaces which have countable tightness, Scientiae Mathematicae Japonicae  \textbf{78} (2015), 201--205.


\bibitem{Gra}
M.~Graev, Free topological groups, Izv. Akad. Nauk SSSR Ser. Mat. \textbf{12} (1948), 278--324 (In Russian). Topology and Topological Algebra. Translation Series 1, \textbf{8} (1962), 305--364.

\bibitem{HM}
K. H. Hofmann, S. A. Morris, \emph{The Structure of Compact Groups: A Primer for the Student -- A Handbook for the Expert}, 3rd edition, de Gruyter, Studies in Mathematics 25, Berlin, 2013.


\bibitem{HuntMorris}
D. C. Hunt, S. A. Morris, Free subgroups of free topological groups, Proc. Second Internat. Conf. Theorey of Groups, Canberra, Lecture Notes in Mathematics 372 (Sptinger, Berlin, 1974), pp. 377--387.


\bibitem{Jar}
H.~Jarchow, \emph{Locally Convex Spaces}, B.G. Teubner, Stuttgart, 1981.

\bibitem{Kakutani}
S. Kakutani, Free topological groups and infinite direct product of topological groups, Proc. Imp. Acad. Tokyo \textbf{20} (1944), 595--598.

\bibitem{KatzMN}
E. Katz, S. A. Morris, P. Nickolas, A free subgroups of the free abelian topological groups on the unit interval, Bull. London Math. Soc. \textbf{14} (1982), 399--402.

\bibitem{Kelley}
J. L. Kelley, \emph{General topology}, D. van Nostrand, New York, 1955.

\bibitem{Kothe}
G.~K\"{o}the, \emph{Topological vector spaces}, Vol. I, Springer-Verlag, Berlin, 1969.

\bibitem{LMP}
A. Leiderman, S. A. Morris, V. Pestov, The free abelian topological group and the free locally convex space on the unit interval, J. London Math. Soc. \textbf{56} (1997), 529--538.


\bibitem{MMO}
J.~Mack, S.~A.~Morris, E.~T.~Ordman, Free topological groups and the projective dimension of a locally compact abelian groups, Proc. Amer. Math. Soc. \textbf{40} (1973), 303--308.

\bibitem{MacLane}
S. MacLane, \emph{Categories for the Working Mathematician}, Springer-Verlag, New York, 1971.

\bibitem{Mar}
A.~A.~Markov,  On free topological groups, Dokl. Akad. Nauk SSSR \textbf{31} (1941), 299--301.

\bibitem{Mich}
E.~Michael,  \emph{$\aleph_0$-spaces}, J. Math. Mech. \textbf{15} (1966), 983--1002.

\bibitem{Morris}
S.A. Morris, Varieties of topological groups and left adjoint functors, J. Austral. Math. Soc. \textbf{16} (1973), 220--227.

\bibitem{Morris1}
S.A. Morris. \emph{Pontryagin duality and the structure of locally compact abelian groups}, Cambridge Univ. Press, Cambridge, 1977.

\bibitem{NaB}
L. Narici, E. Beckenstein,  \emph{Topological vector spaces}, Second Edition, CRC Press, New York, 2011.

\bibitem{Nickolas}
P. Nickolas, A Kurosh subgroup theorem for topological groups, Proc. London Math. Soc. \textbf{42} (1981), 461--477.


\bibitem{OrdThomas}
E. T. Ordman, B. V. Smith-Thomas, Sequential conditions and free topological groups, Proc. Amer. Math. Soc. \textbf{79} (1980), 319--326.


\bibitem{Prot}
I.~Protasov, Maximal vector topologies, Topology Appl.  \textbf{159} (2012), 2510--2512.

\bibitem{ZP2}
I.~V.~Protasov, E.~G.~Zelenyuk, \emph{Topologies on groups determined by sequences}, Monograph Series, Math. Studies VNTL, L'viv, 1999.

\bibitem{Rai}
D.~A.~Ra\u{\i}kov, Free locally convex spaces for uniform spaces, Math. Sb. \textbf{63} (1964), 582--590.

\bibitem{Sa}
S.~A.~Saxon, Nuclear and product spaces, Baire-like spaces, and the strongest locally convex topology, Math. Ann. \textbf{197} (1972), 87--106.

\bibitem{Sipacheva}
O. V. Sipacheva, The topology of a free topological group, (Russian) Fundam. Prikl. Mat. \textbf{9} (2003), no. 2, 99--204; translation in J. Math. Sci. (N. Y.) \textbf{131} (2005), no. 4, 5765--5838.

\bibitem{STh}
B.~V.~Smith-Thomas, Free topological groups, General Topology Appl. \textbf{4} (1974), 51--72.

\bibitem{Ste}
N. E. Steenrod, A convenient category of topological spaces, Michigan Math. J. {\bf 14} (1967), 133--152.

\bibitem{Tkac}
M.~G.~Tkachenko, On completeness of free abelian topological groups, Soviet Math. Dokl. \textbf{27} (1983), 341--345.

\bibitem{Usp}
V.~V.~Uspenski\u{\i}, On the topology of free locally convex spaces, Soviet Math. Dokl. \textbf{27} (1983), 781--785.

\bibitem{Usp2}
V.~V.~Uspenski\u{\i}, Free topological groups of metrizable spaces,  Math. USSR-Izv. \textbf{37} (1991), 657--680.



\end{thebibliography}

\end{document}